\documentclass{amsart}
\usepackage{etex}
 \usepackage[foot]{amsaddr}
\usepackage{graphicx}
\usepackage{epstopdf}

\usepackage{cite}
\usepackage{tikz, tikz-cd, tkz-graph}

\usepackage{xcolor}

\usepackage{hyperref}
\hypersetup{
  colorlinks,
  linkcolor={red!50!black},
  citecolor={blue!50!black},
  urlcolor={blue!80!black}
}

\usepackage[english]{babel}
\usepackage{amsmath}
\usepackage{amssymb}
\usepackage{amsthm}
\usepackage{xypic}
\usepackage{longtable}
\usepackage{array}

 \usepackage[enableskew]{youngtab}

\usepackage{epsfig}
\usepackage{hyperref}
\usepackage{enumitem}
\usepackage{booktabs}
 \usepackage{longtable} 
 \usepackage{pdflscape} 
 \usepackage{colortbl} 
 \usepackage{arydshln} 
 \usepackage{calc} 
\usepackage[vcentermath]{genyoungtabtikz}
\Yboxdim{6pt}

\graphicspath{{./images/}}

\newtheorem{thm}{\bf Theorem}[section]
\newtheorem{eg}[thm]{\bf Example}
\newtheorem{prop}[thm]{\bf Proposition}

\newtheorem{rem}[thm]{\bf Remark}
\newtheorem{mydef}[thm]{\bf Definition}
\newtheorem{lem}[thm]{\bf Lemma}

\newcommand{\CC}{\mathbb{C}}

\newcommand{\ZZ}{\mathbb{Z}}

\newcommand{\U}{\"u}

\DeclareMathOperator{\Fl}{\mathrm{Fl}}

\DeclareMathOperator{\Gr}{\mathrm{Gr}}
\DeclareMathOperator{\Sym}{\mathrm{Sym}}

\DeclareMathOperator{\rk}{\mathrm{rank}}

\DeclareMathOperator{\Pic}{\mathrm{Pic}}

\renewcommand{\emptyset}{\varnothing}

\newcolumntype{C}[1]{>{\centering\let\newline\\\arraybackslash\hspace{0pt}}m{#1}}
\title[A Pl\"ucker coordinate mirror for type A flag varieties]{A Pl\"ucker coordinate mirror \\ for type A flag varieties}

\author{Elana Kalashnikov}
\address{{Department of Mathematics \\ 
Harvard University \\ 
Cambridge, MA 02138}}
\address{{Faculty of Mathematics \\ 
Higher School of Economics \\ 
Moscow, Russia}}
\email{kalashnikov@math.harvard.edu}

\begin{document}
\begin{abstract}
We introduce a superpotential for partial flag varieties of type~$A$. This is a map $W: Y^\circ \to \CC$, where $Y^\circ$ is the complement of an anticanonical divisor on a product of Grassmannians.  The map $W$ is expressed in terms of Pl\U cker coordinates of the Grassmannian factors. This construction generalizes the Marsh--Rietsch Pl\U cker coordinate mirror for Grassmannians. We show that in a distinguished cluster chart for $Y$, our superpotential agrees with earlier mirrors constructed by Eguchi--Hori--Xiong and Batyrev--Ciocan-Fontanine--Kim--van Straten. Our main tool is quantum Schubert calculus on the flag variety.  
\end{abstract}

\maketitle

\section{Introduction}\label{sec:intro}
In this paper, we use quantum Schubert calculus to propose a mirror partner -- a superpotential -- for type A flag varieties that generalizes the Pl\U cker coordinate superpotential for Grassmannians introduced by Marsh--Rietsch \cite{MarshRietsch}. Schubert classes on a partial flag variety  $\Fl(n;r_1,\dots,r_\rho)$  are indexed by tuples of partitions. We can associate an element of the coordinate ring of the product of Grassmannians $Y=\prod_{i=1}^\rho \Gr(r_{i-1},r_{i-1}-r_i)$ to every Schubert class. We write down a special sum of ratios of Schubert classes expressing the anti-canonical class of the flag variety, and re-interpret it as a rational function on $Y$. This is the Pl\U cker coordinate mirror for the flag variety.  

Batyrev--Ciocan-Fontanine--Kim--van Straten~\cite{flagdegenerations} have also constructed a conjectural superpotential for a flag variety, by taking the Gelfand--Cetlin toric degeneration of the flag variety and looking at the Hori--Vafa mirror of the degenerate toric variety. The BCFKvS superpotential generalizes earlier mirrors proposed by Eguchi--Hori--Xiong (for the Grassmannian)\cite{eguchi} and Givental \cite{giventalflag} (for full flag manifolds). The authors of \cite{flagdegenerations} conjecture that the period of this superpotential agrees with the quantum period (the top term of Givental's J function) of the flag variety. Such superpotentials are called \emph{Laurent polynomial mirrors} and conjectured to exist for all Fano varieties. 

The main result of this paper is that in a particular cluster chart, the Pl\U cker coordinate superpotential and the BCFKvS superpotential  -- which are obtained in very different ways -- are isomorphic. 

\begin{thm}\label{thm:first} Let $\Fl(n;r_1,\dots,r_\rho)$ be a flag variety, and $W_P$ the Pl\U cker coordinate superpotential as defined in Definition \ref{def:mirror}.  The pullback of $W_P$ along the rectangles cluster chart (see Definition \ref{def:rectangles}) is isomorphic to the Batyrev--Ciocan-Fontanine--Kim--van Straten mirror.
\end{thm}
The enriched structure of the Grassmannian Pl\U cker coordinate superpotential, arising from the cluster structure of the Grassmannian, allows Marsh--Rietsch \cite{MarshRietsch} to prove a very strong mirror statement.  Their results in particular imply the conjecture of Batyrev--Ciocan-Fontanine--Kim--van Straten~\cite{flagdegenerations} for Grassmannians. This gives one of the few proven examples of Laurent polynomial mirrors outside of the toric context.  We hope that the Pl\U cker coordinate mirror proposed here will open up a way of proving the conjecture for partial flag varieties. This result would have significant implications for the Fano classification program, which seeks to classify Fano varieties by their mirror partners: it would extend the universe of proven mirrors from complete intersections in Grassmannians and toric varieties to flag varieties.

\paragraph{\emph{Plan of the paper}} In \S \ref{sec:summary}, we explain more carefully the relation of the Pl\U cker coordinate mirror to other mirror constructions in the literature and the Fano classification program, and summarize the construction of the Pl\U cker coordinate mirror. In \S \ref{sec:quantumcohomology} we review the basics of quantum cohomology and quantum Schubert calculus, and recall the results we will need. In \S \ref{sec:construction}, we use quantum Schubert calculus on the flag variety to define the Pl\U cker coordinate superpotential of the flag variety. In \S \ref{sec:comparison}, we prove the above theorem. 
In the last section, \S \ref{sec:examples} we look closer at a family of examples, discussing the relationship in this case with the Gu--Sharpe proposal. 
\subsection*{Acknowledgments}
The author is very grateful for helpful conversations with Tom Coates, Wei Gu, Konstanze Rietsch, and Lauren Williams.

\section{Summary and context}\label{sec:summary}
In addition to the Pl\U cker coordinate and BCFKvS superpotentials, there are a number of other mirror constructions for flag varieties in the literature, corresponding to the different flavours of constructions of flag varieties. Grassmannians and type A flag varieties can be constructed either as homogeneous spaces $G/P$ or as GIT quotients $V//G$. For any homogeneous space $G/P$,  there is the Lie theoretic mirror proposed by Rietsch \cite{rietsch2006}. She showed via the Peterson variety that the critical locus of the superpotential computes the quantum cohomology ring of $G/P$.  

Alternatively, viewing a flag variety as a GIT quotient $V//G$, one can apply the Abelian/non-Abelian correspondence. The Abelian/non-Abelian correspondence is a powerful and quite general tool that allows one to replace $G$ with a maximal torus $T\subset G$, then bring to bear the power of mirror constructions for toric varieties $V//T$ \cite{horivafa,twoproofs,CiocanFontanineKimSabbah2008,Webb2018,gusharpe, gukalashnikov}. Gu--Sharpe \cite{gusharpe} propose a mirror for $V//G$ which in particular is a Laurent polynomial in $\dim(V//T)$ variables with $\rk \Pic(V//G)$ quantum parameters. The Jacobi ring of the mirror computes the quantum cohomology of $V//G$ when $V//G$ is a quiver flag variety \cite{gukalashnikov}. 

Type A flag varieties are also Fano varieties. Conjecturally, under mirror symmetry $n$-dimensional Fano varieties up to deformation correspond to certain Laurent polynomials in $n$-variables up to mutation \cite{kasprzyktveiten}. Proving this conjecture would have major implications for the Fano classification program, making the Fano classification problem entirely combinatorial. A Laurent polynomial is said to be \emph{mirror} to a Fano variety in this sense if the \emph{regularized quantum period} of the Fano variety is equal to the \emph{classical period} of the Laurent polynomial. The quantum period of a Fano variety is an invariant built out of genus 0 Gromov--Witten invariants; it is a piece of Givental's J function. See \cite{fanomanifolds} for a good introduction to this perspective. We call these mirrors \emph{Laurent polynomial mirrors}. Laurent polynomial mirrors are expected for any Fano variety. The Gu--Sharpe mirror is not a mirror in this sense as it has too many variables: $\dim(V//T)$ instead of $\dim(V//G)$.

The Batyrev--Ciocan--Fontanine--Kim--van Straten superpotential is a Laurent polynomial mirror in this sense. Laurent polynomial mirrors can be constructed for any smooth Fano toric variety and most Fano toric complete intersections. More generally, a toric degeneration of a Fano variety can allow one to produce a conjectural Laurent polynomial mirror. In \cite{flagdegenerations}, they prove that there is a toric degeneration of the flag variety to the toric variety associated to the Newton polytope of the BCFKvS mirror. 

In the case of the Grassmannian, the Pl\U cker coordinate mirror \cite{MarshRietsch} sits between the Lie theoretic mirror and the BCFKvS mirror. It is isomorphic to the Lie theoretic mirror, but much simpler and easier to work with.  The connection between the Pl\U cker coordinate mirror and the Gu-Sharpe proposal largely remains mysterious; however, in Example \ref{eg:grassmanniancrit}, we explain how to explicitly identify the critical locus of the Gu--Sharpe mirror and the Pl\U cker coordinate mirror of the Grassmannian. 

The Pl\U cker coordinate superpotential for $\Gr(n,r)$ is a Laurent polynomial in the Pl\U cker coordinates of the dual Grassmannian. Pl\U cker coordinates are redundant, so like the Gu--Sharpe mirror it has too many variables to be a mirror in the sense of Fano classification program. However, the cluster structure  of the coordinate ring of the Grassmannian \cite{scott2006} gives a way to produce many Laurent polynomial mirrors to the Grassmannian. Every cluster chart (certain algebraically independent sets of Pl\U cker coordinates) gives an expansion of the Pl\U cker coordinate mirror into a Laurent polynomial in the correct number of variables. These Laurent polynomials are related by quiver mutations. There is a distinguished cluster chart, called the rectangles cluster chart, where Marsh--Rietsch show that the Pl\U cker coordinate superpotential expands into a Laurent polynomial isomorphic to the BCFKvS mirror. As a corollary of their powerful results for the Pl\U cker coordinate superpotential, they prove the conjecture of \cite{flagdegenerations} in the Grassmannian case: they show that the period of the BCFKvS mirror agrees with the quantum period of the Grassmannian. This can even be extended to complete intersections in Grassmannians \cite{Przyjalkowski_2015}.

We propose an analogue of the Pl\U cker coordinate superpotential for type A flag varieties, using quantum Schubert calculus of the flag variety.  We now briefly summarize its construction. To do this, we need to take a closer look at the Grassmannian case. Let $\Gr(n,r)$ denote the Grassmannian of $r$-dimensional quotients of $\CC^n$. The quantum cohomology ring of the Grassmannian is a quotient of the symmetric polynomial ring in $r$ variables, with coefficients in $\CC[q]$, where $q$ is the quantum parameter. Schur polynomials $s_\lambda$ span the ring of symmetric polynomials; they are indexed by partitions $\lambda$ with length at most $r$.  An additive basis for the quantum cohomology ring of the Grassmannian are the Schur polynomials indexed by partitions $\lambda$ of length at most $r$ satisfying $\lambda_1 \leq n-r$. We denote these partitions $S(n,r)$. These same partitions index the Pl\U cker coordinates of the Grassmannian. 

 The Pl\U cker coordinate superpotential the Grassmannian $\Gr(n,r)$ shares a common structure with the mirror proposed by Gu--Sharpe: both can be interpreted as a sum of $n$ copies of the hyperplane class in the Grassmannian. Marsh--Rietsch take $n$ equations of the form 
 \[s_{\yng(1)} * s_\lambda=q^i s_{\mu}\]
 where $i=0,1$ depending on the partition. Here $\lambda=(a,\dots,a)$ is a rectangular partition either maximally wide or maximally tall: we denote the set of such partitions $M(n,r)$. After localizing, the sum 
 \[ \sum_{\lambda \in M(n,r)} \frac{q^{i} s_\mu}{s_\lambda}\]
is equal to $n s_{\yng(1)}=-K_{\Gr(n,r)}$, the anti-canonical class.  The sum produces a Laurent polynomial in the Pl\U cker coordinates of the dual Grassmannian $\Gr(n,n-r)$ by replacing $s_\lambda$ with the associated Pl\U cker coordinate. So, for example, the Marsh--Riestch Pl\U cker coordinate superpotential for $\Gr(4,2)$ is
\[\frac{p_{\yng(1)}}{p_{\emptyset}}+\frac{p_{\yng(2,1)}}{p_{\yng(2)}}+\frac{p_{\yng(2,1)}}{p_{\yng(1,1)}}+\frac{q p_{\yng(1)}}{p_{\yng(2,2)}}.\]
In \cite{MarshRietsch}, the open subvariety of the dual Grassmannian on which the Pl\U cker coordinate superpotential is a function (i.e. where $p_\lambda \neq 0$, $\lambda \in M(n,r)$) is denoted $\Gr(n,n-r)^\circ$. 
\paragraph{\emph{The Pl\U cker coordinate superpotential for a flag variety}}
The two main ingredients of the Pl\U cker coordinate superpotential of the Grassmannian are, firstly,  a way of writing the anti-canonical class as a sum of ratios of Schur polynomials, and secondly, a way of associating to a Schur polynomial a function on another Grassmannian. 

To generalize this to type A flag varieties, we use Schubert classes and quantum Schubert calculus. Let $\Fl(n;r_1,\dots,r_\rho)$ be the flag variety of quotients of $\CC^n$.   The detailed description of the first ingredient -- a way of writing the anti-canonical class as a sum of ratios of Schubert classes -- is in \S \ref{sec:construction}. For the second ingredient, we use a less common way of indexing Schubert classes: we index Schubert classes by tuples of partitions $\underline{\mu}=(\mu_1,\dots,\mu_\rho)$, where $\mu_i \in S(r_{i-1},r_{i})$. Observe that each such tuple  defines a function $p_{\underline{\mu}}$  on the product of Grassmannians $Y=\prod_{i=1}^\rho \Gr(r_{i-1},r_{i-1}-r_i)$: $p_{\underline{\mu}}=\prod_{i=1}^\rho p^i_{\mu_i}$, where $p^i_\lambda$ is the pullback to $Y$ of the Pl\U cker coordinate $p_\lambda$ on the $i^{th}$ Grassmannian factor. 

In Definition \ref{def:mirror}, we use the breakdown of the anti-canonical class to write down a Laurent polynomial in the Pl\U cker coordinates of $Y$, which is a function on
\[Y:=\prod_{i=1}^\rho \Gr(r_{i-1},r_{i-1}-r_i)^\circ. \]
\begin{eg} For the flag variety $\Fl(4;2,1)$,
\[W_P:=\frac{p^1_{\yng(1)}}{p^1_{\emptyset}}+\frac{p^1_{\yng(2,1)}+q_1}{p^1_{\yng(2)}}+\frac{p^1_{\yng(2,1)}}{p^1_{\yng(1,1)}}+\frac{q_1 p^1_{\yng(1)} p^2_{\yng(1)}}{p^1_{\yng(2,2)}}+\frac{p^2_{\yng(1)}}{p^2_{\emptyset}}+\frac{q_2}{p^2_{\yng(1)}}.\]
\end{eg}
 As for the Grassmannian, the Pl\U cker coordinate mirror of the flag variety gives many different conjectural Laurent polynomial mirrors. The cluster structure involved in this proposal is not the cluster structure of the flag variety or its Langlands dual, but rather that of a product of Grassmannians, $Y$.  On the mirror side, there is no non-trivial fibration structure; unlike like the flag variety, $Y$ is just a product. This structure is instead encoded in the superpotential: $W_P$ is almost the mirror of $Y$, except with certain correction terms arising from the non-trivial product structure of the flag variety. These correction terms keep track of the difference between using the quantum product to multiply in $Y$ or the flag variety.  

\section{Quantum cohomology and quantum Schubert calculus}\label{sec:quantumcohomology}
\subsection{Gromov--Witten invariants and quantum cohomology}
We briefly review small quantum cohomology. Let $Y$ be a smooth projective variety. Given $n \in \ZZ_{\geq 0}$ and $\beta \in H_2(Y)$, let  $M_{0,n}(Y,\beta)$ be the moduli space of genus zero stable maps to $Y$ of class $\beta$, and with $n$ marked points~\cite{Kontsevich95}. While this space may be highly singular and have components of different dimensions, it has a \emph{virtual fundamental class} $[M_{0,n}(Y,\beta)]^{virt}$ of the expected dimension~\cite{BehrendFantechi97,LiTian98}. There are natural evaluation maps $ev_i: M_{0,n}(Y,\beta) \to Y$ taking the class of a stable map $f: C \to Y$ to $f(x_i)$, where $x_i \in C$ is the $i^{th}$ marked point. There is also a line bundle $L_i \to M_{0,n}(Y,\beta)$ whose fiber at $f: C \to Y$ is the cotangent space to $C$ at $x_i$. The first Chern class of this line bundle is denoted $\psi_i$. Define:
\begin{equation}
  \label{eq:GW}
  \langle \tau_{a_1}(\alpha_1),\dots,\tau_{a_n}(\alpha_n) \rangle_{n,\beta} = \int_{[M_{0,n}(Y,\beta)]^{virt}} \prod_{i=1}^n ev_i^*(\alpha_i) \psi_i^{a_i}
\end{equation}
where the integral on the right-hand side denotes cap product with the virtual fundamental class.  If $a_i=0$ for all $i$, this is called a (genus zero) Gromov--Witten invariant and the $\tau$ notation is omitted; otherwise it is called a descendant invariant. It is deformation invariant. 

Now suppose that $Y$ is a smooth Fano variety.  The \emph{quantum cohomology ring} is defined by giving a deformation of the usual cup product of $H^*(Y)$ for every $t \in H^*(Y)$. The structural constants defining the new product are given by Gromov--Witten invariants.

Let $\{T_i\}$ be a homogeneous basis of $H^*(Y,\CC))$ and $\{T^i\}$ a dual basis. Let $t \in H^2(Y,\CC)$. The small quantum product is defined by 
\[\langle T^a \circ_t T^b, T^c\rangle:=\sum_{d \in H_2(Y)} e^{\int_{d}t} \int_{[M_{0,3}(Y,d)]^{virt}} ev_1^*(T^a) ev_2^*(T^b) ev_3^*(T^c).\]
The fact that $Y$ is Fano ensures that this sum is finite. 

  If $T_1,\dots,T_r$ are a basis of $H^2(Y,\CC)$, and $t_i$ a parameter for $T_i$, define $q_i:=e^{t_i}.$ For $d=\sum_{i=1}^r d_i T_i$, write $q^d=q_1^{d_1} \cdots q_r^{d_r}$. We can re-write the above as
\[\langle T^a \circ T^b, T^c\rangle:=\sum_{d \in H_2(Y)} q^d \int_{[M_{0,3}(X,d)]^{virt}} ev_1^*(T^a) ev_2^*(T^b) ev_3^*(T^c).\]
We view the quantum product as giving a product structure on the cohomology of $H^*(X)$ with coefficients in the polynomial ring in the $q_i$. 
\subsection{Quantum cohomology of flag varieties}
Let $\Fl(n;r_1,\dots,r_\rho)=\Fl(n;\underline{r})$ denote the flag variety of quotients: it parametrizes quotients \[\CC^n \to V_1 \to \cdots \to V_\rho \to 0\] where $V_i$ has dimension $r_i$. It is equipped with $\rho$ tautological quotient bundles of ranks $r_1,\dots,r_\rho$. 

Schubert classes are geometrically described cohomology classes in flag varieties; they give free generators for the cohomology of the flag variety. Schubert classes on the flag variety of quotients $\Fl(n;r_1,\dots,r_\rho)$ can be indexed by permutations (as is most common), but also as tuples of partitions. This is the indexing set we use. 

Recall from the introduction that $S(n,r)$ denotes the partitions generating the cohomology of the Grassmannian -- that is, partitions fitting inside of an $r \times (n-r)$ box. There is a Schubert class associated to each set of partitions $\underline{\mu}=(\mu_1,\dots,\mu_\rho)$, where $\mu_i \in S(r_{i-1},r_{i})$. We set $r_0:=n$. We set $S(n,\underline{r})$ to be this set of tuples of $\rho$ partitions:
\[S(n,\underline{r}):=\prod_{i=1}^\rho S(r_{i-1},r_{i}).\]
We denote the class of the Schuberty variety in the quantum cohomology ring corresponding to $\underline{\mu}$ as $s_{\underline{\mu}}$. For a partition $\lambda$, the transpose partition is denoted $\lambda^t$.

For $\lambda \in S(r_{i-1},r_i)$, we set $s^i_\lambda$ to be the special Schubert class corresponding to $\underline{\mu}$, where $\mu_i=\lambda$ and otherwise $\mu_j=\emptyset$. The class $s^i_\lambda$ is a symmetric polynomial in the Chern roots of the rank $r_i$ tautological quotient bundle on $\Fl(n;r_1,\dots,r_\rho)$. The class $s^i_{\yng(1)}$ is the first Chern class of the rank $r_i$ tautological quotient bundle. We extend this notation slightly: if $\lambda \in S(r_{i-1},r_i)$ and $\lambda' \in S(r_{j-1},r_j)$ for $i \neq j$, we write $s^{i,j}_{\lambda,\lambda'}$ for the Schubert class corresponding to $\underline{\mu}$, where $\mu_i=\lambda,$ $\mu_j=\lambda'$, and otherwise $\mu_k=\emptyset$. We will often consider partitions $\lambda$ where $\lambda$ is rectangular. We denote the class associated to the partition which is an $a \times b$ grid of boxes as $s^i_{a \times b}.$

The quantum cohomology of the flag variety can be written as \[\CC[x_{ij},q_i: 1 \leq i \leq \rho, 1 \leq j \leq r_i]^{\Sym_{r_1} \times \cdots \times \Sym_{r_\rho}}/I\] for an explicitly described $I$ \cite{astashkevich1995,kim}.  The $x_{i1},\dots,x_{ir_i}$ are the Chern roots of the tautological quotient bundles. Here $\Sym_{r_i}$ is the symmetric group which acts on the polynomial ring by permuting the $x_{i1},\dots,x_{i r_i}$, and $I$ is an explicitly described ideal. A general description for the quantum product of two Schubert classes as a sum of Schubert classes is not known, except in special cases.  The expansion of 
\[s^i_{\yng(1)} * s_{\underline{\mu}}\] 
is determined by the quantum Pieri rule of Ciocan--Fontanine \cite{ciocanpieri}.  We will need this rule in the special case where $s_{\underline{\mu}}=s^i_\lambda$, where $\lambda \in M(r_{i-1},r_i)$. Recall from the introduction that
\[M(n,k):=\{\lambda=(a,\dots,a) \in S(n,k) \mid \text{ $\lambda$ is a maximally wide or tall rectangle}\}.\] 
Lemma \ref{lem:formula} records the cases we need in the notation we have been using. In the lemma, we set $s^{0,i}_{\lambda,\lambda'}=s^{i,\rho+1}_{\lambda',\lambda}=0$ unless $\lambda=\emptyset$, in which case  we set them both to $s^i_{\lambda'}$.
\begin{lem}\label{lem:formula} The following equations hold in the quantum cohomology ring of the flag variety $\Fl(n;r_1,\dots,r_\rho)$:
\begin{enumerate}
\item If $a<r_{i-1}-r_i$, then 
\[s^i_{\yng(1)} *s^i_{r_i \times a}=s^i_{\{\underbrace{a+1,a,\dots,a}_{r_i \text{ parts}}\}}.\]
\item If $a<r_i-r_{i+1}$, then
\[s^i_{\yng(1)}*s^i_{a \times (r_{i-1}-r_i)}=s^i_{\{\underbrace{r_{i-1}-r_i,\dots,r_{i-1}-r_i,1}_{a+1 \text{ parts}}\}}+s^{i-1,i}_{\yng(1), a \times (r_{i-1}-r_i)}.\]
\item If $r_i-r_{i+1}\leq a<r_i$, then
\begin{align*}
s^i_{\yng(1)}*s^i_{a \times (r_{i-1}-r_i)}&\\&=s^i_{\{\underbrace{r_{i-1}-r_i,\dots,r_{i-1}-r_i,1}_{a+1 \text{ parts}}\}}\\&+s^{i-1,i}_{\yng(1), a \times (r_{i-1}-r_i)}+q_i s^{i,i+1}_{(a-1)\times (r_{i-1}-r_i-1),(a-(r_i-r_{i+1}))^t} .
\end{align*}
\item If $a=r_i$, then
\[s^i_{\yng(1)}*s^i_{r_i \times( r_{i-1}-r_i)}=s^{i-1,i}_{\yng(1), a \times (r_{i-1}-r_i)}+q_i s^{i,i+1}_{(a-1)\times (r_{i-1}-r_i-1),(a-(r_i-r_{i+1}))^t} .\]
\end{enumerate}
\end{lem}
This follows immediately from \cite{ciocanpieri}, after translating between the different notation. We give a quick sketch of how to see this, for completeness.
\paragraph{\emph{Changing notation}} Let $\Sym(r_1,\dots,r_\rho)$ be the subset of $\Sym_n$ with descents in $\{r_1,\dots,r_\rho\}$. These are the permutations indexing the Schubert classes of $\Fl(n;\underline{r})$.  Given an element of $\Sym(r_1,\dots,r_\rho)$, in word representation denoted as $a_1\dots a_n$, we write down a tuple of partitions $(\lambda_1,\dots,\lambda_\rho)$. For each $i=0,\dots,\rho$, let $T_i:=\{a_1,\dots,a_{r_i}\}$, which we can re-label as $\{b_1<\cdots<b_{r_i}\}$. If $T_{i+1}=\{b_{k_1},\dots,b_{k_{r_{i+1}}}\}$, then we consider the subset 
\[\{k_1,\dots,k_{r_{i+1}}\} \subset \{1,\dots,r_i\}.\] 
This subset determines a partition in $S(r_{i},r_{i+1})$: it is the partition that, when viewed as giving a path from the bottom left to the top right of an $r_{i+1} \times (r_i-r_{i+1})$ grid, has vertical steps at $k_1,\dots,k_{r_{i+1}}$. 

\begin{eg} Consider $\Fl(4;2,1)$. The identity permutation $1 2 3 4$ gives tuple of partitions $(\emptyset,\emptyset)$. The permutation $2 1 3 4$ gives the tuple $(\emptyset,\yng(1))$. The permutation $1 3 2 4$ gives the tuple $(\yng(1),\emptyset)$. The permutation $3 2 1 4$  corresponds to $(\yng(1,1),\yng(1))$. 
\end{eg}
In general, the transposition $(r_i r_i+1)$  is $s^i_{\yng(1)}$. 
\begin{proof}[Proof of Lemma \ref{lem:formula}]
It is easy to check the classical terms of the formula. We sketch how to obtain the quantum corrections. 
 
 Fix $i$, and let $s_i:=(r_i r_i+1)$. Let $\lambda:=a\times (r_{i-1}-r_i)$. The permutation for $\lambda$ is
\[\omega:=1 2 \dots r_i-a\hspace{3mm} r_{i-1}-a+1 \dots r_{i-1}\hspace{3mm} r_i-a+1 \dots r_{i-1}-a \hspace{3mm} r_{i-1}+1 \dots n.\]
We now apply Quantum Pieri. Note that in Ciocan--Fontanine's notation in \cite{ciocanpieri},  $s^i_{\yng(1)}=\alpha_{1,\rho-i+1}$.

The classical terms are easily determined, so we just consider the quantum corrections. There are possible quantum corrections for each $h,l$ satisfying $\rho \geq h \geq i \geq l \geq 1$ such that
\[\omega(r_{h}) >\max\{\omega(r_h+1),\dots,\omega(r_{l-1})\}.\]
It's easy to see that only $h=l=i$ satisfies this condition. This contributes a term
$q_i [X_{\omega'}]$
where $\omega'=\omega s_{r_i} \cdots s_{r_{i-1}-1} s_{r_{i-1}} \cdots s_{r_{i+1}+1}$ provided this permutation lies in $\Sym(r_1,\dots,r_\rho)$ and has the correct number of transpositions. If $r_{i+1}=r_i-1$, then 
\[\omega'=1 \dots r_i-a \hspace{3mm}  r_{i-1}-a+1 \dots r_{i-1}-1 \hspace{3mm} r_i-a+1 \dots r_{i-1}-a \hspace{2mm} r_{i-1} \dots n. \]
It is easy to see that this is precisely the quantum correction prescribed in the lemma. More generally, if $r_{i+1}<r_i-a$ then $\omega'$ does not satisfy the conditions, and otherwise it is the permutation described by the partitions in third case of the lemma. 
\end{proof}

\section{Construction of the superpotential}\label{sec:construction}
\subsection{Construction}
Lemma \ref{lem:formula} allows us to mimic the Marsh--Rietsch construction of the Pl\U cker coordinate superpotential of the Grassmannian. For each $i$, consider for each of the $r_{i-1}$ elements of $M(r_{i-1},r_i)$ the expansion in the quantum Schubert basis 
\[F^i_\lambda:=s^i_{\yng(1)}*s^i_\lambda.\]
The expansion is described by Lemma \ref{lem:formula}; $F^i_\lambda$ is a sum of Schubert classes with coefficients in $\CC[q_1,\dots,q_\rho]$. 

The sum
\[\sum_{i=1}^\rho (\sum_{\lambda \in M(r_{i-1},r_i)} \frac{F^i_\lambda}{s^i_{\lambda}})-r_{i+1} s^i_{\yng(1)}.\]
is equivalent to $\sum_{i=1}^\rho (r_{i-1}-r_{i+1}) s^i_{\yng(1)}$ which is the anti-canonical class of the flag variety.

 The mirror of the flag variety should be a Laurent polynomial in the coordinates of an appropriate space: in the Grassmannian case, Marsh--Rietsch use the fact that $S(n,r)$ indexes both the basis of the cohomology of the Grassmannian and the Pl\U cker coordinates of dual Grassmannian. 

The set $S(n,\underline{r})$ naturally indexes elements of the coordinate ring of the product of Grassmannians $Y(n,\underline{r}):=\prod_{i=1}^\rho \Gr(r_{i-1},r_{i-1}-r_i)$.
\begin{rem}
Alternatively, we can consider these to be coordinates on the fibers of the flag variety viewed as a tower of Grassmannian bundles.
\end{rem}
Let $Q_i$ be the tautological quotient bundle pulled back to $Y(n,\underline{r})$ from the $i^{th}$ Grassmannian factor. Sections of $\det(Q_i)$ are indexed by $\lambda \in S(r_{i-1},r_i)$. We write $p^i_\lambda$ for the Pl\U cker coordinate associated to $i$ and $\lambda$. 

We denote $Y(n,\underline{r})^\circ:=\prod_{i=1}^\rho \Gr(r_{i-1},r_{i-1}-r_i)^\circ$ the locus in $Y(n,\underline{r})$ where $p^i_\lambda \neq 0$ for all $i$ and $\lambda \in M(r_{i-1},r_i)$. This is the complement of an anti-canonical divisor on $Y(n,\underline{r})$. 

To each Schubert class $s_{\underline{\mu}}$ we associate the product
\[p_{\underline{\mu}}:=\prod_{i=1}^\rho p^i_{\mu_i}.\]
We denote the polynomial in the coordinate ring of $Y(n,\underline{r})$ and the $q_1,\dots,q_\rho$ obtained by replacing the Schubert classes in $F^i_\lambda$ with Pl\U cker coordinates in this way as $G^i_\lambda$. 
\begin{mydef}\label{def:mirror} The Pl\U cker coordinate superpotential $W_P$ of the flag variety is
\[\sum_{i=1}^\rho (\sum_{\lambda \in M(r_{i-1},r_i)} \frac{G^i_\lambda}{p^i_{\lambda}})-r_{i+1} p^i_{\yng(1)}.\]
\end{mydef}
The superpotential of the flag variety $\Fl(n;r_1,\dots,r_\rho)$ can be viewed as a regular function $Y(n,\underline{r})^\circ \times \CC^\rho \to \CC$. 
\begin{rem} Note that after cancellations, all coefficients are positive in the expression of the Pl\U cker coordinate superpotential. The Pl\U cker coordinate superpotential is very close to the sum of the mirrors of each of the Grassmannian factors (which corresponds to the mirror of the product of Grassmannians) -- in particular the denominators are the same. The extra terms should be thought of keeping track of the non-trivial fiber structure of the flag variety.  
\end{rem}
\begin{eg} Consider the flag variety $\Fl(6;4,2,1)$. The Pl\U cker coordinate superpotential is
\begin{align*}
\frac{p^1_{\yng(1)}}{p^1_\emptyset}+\frac{p^1_{\yng(2,1,1,1)}}{p^1_{\yng(1,1,1,1)}}+\frac{p^1_{\yng(2,1)}}{p^1_{\yng(2)}}+\frac{p^1_{\yng(2,2,1)}+q_1 p^1_{\yng(1)}}{p^1_{\yng(2,2)}}+\frac{p^1_{\yng(2,2,2,1)}+q_1 p^1_{\yng(1,1)} p^2_{\yng(1)}}{p^1_{\yng(2,2,2)}}+\frac{q_1 p^1_{\yng(1,1,1)} p^2_{\yng(1,1)}}{p^1_{\yng(2,2,2,2)}}
\\+\frac{p^2_{\yng(1)}}{p^2_\emptyset}+\frac{p^2_{\yng(2,1)}}{p^2_{\yng(1,1)}}+\frac{p^2_{\yng(2,1)}+q_2}{p^2_{\yng(2)}}+\frac{q_2 p^2_{\yng(1)} p^3_{\yng(1)}}{p^2_{\yng(2,2)}}+\frac{p^3_{\yng(1)}}{p^3_{\emptyset}}+\frac{q_3}{p^3_{\yng(1)}}.
\end{align*}
\end{eg}

\subsection{Laurent polynomials from $W_P$}
The coordinate ring of the Grassmannian $\Gr(n,r)$ and the open subvariety $\Gr(n,r)^\circ$ has a cluster structure \cite{scott2006}. There are distinguished cluster charts $(\CC^*)^{r(n-r)} \to \Gr(n,r)^\circ$; these correspond to certain sets, called cluster seeds, of algebraically independent collections of Pl\U cker coordinates. All Pl\U cker coordinates can be expanded as Laurent polynomials (with positive coefficients) in the Pl\U cker coordinates of a given cluster seed. Every cluster seed contains the \emph{frozen variables}: these are Pl\U cker coordinates corresponding to $M(n,r)$. 
\begin{eg} The set $p_{\emptyset},p_{\yng(1)}, p_{\yng(1,1)}, p_{\yng(2)},p_{\yng(2,2)}$ is a cluster for $\Gr(4,2)$.
\end{eg}
This is an example of a seed with plays an important role in the cluster structure of the Grassmannian.
\begin{mydef}\label{def:rectangles}The \emph{rectangles cluster} is the cluster whose Pl\U cker coordinates are the rectangular partitions $i \times j$ in $S(n,r)$. 
\end{mydef}
Given any cluster seed, pulling the Pl\U cker coordinate superpotential $W_P$ of the Grassmannian back via the cluster chart results in a Laurent polynomial in the Pl\U cker coordinates of the cluster seed. The coefficients of this Laurent polynomial are always positive.  
\begin{eg} Pulling back the superpotential of $\Gr(4,2)$ via the rectangles cluster gives 
\[\frac{p_{\yng(1)}}{p_\emptyset}+\frac{p_{\yng(1,1)}}{p_{\yng(1)}}+\frac{p_{\yng(2)}}{p_{\yng(1)}}+\frac{p_{\yng(2,2)}}{p_{\yng(1)}p_{\yng(1,1)}}+\frac{p_{\yng(2,2)}}{p_{\yng(1)}p_{\yng(2)}}+\frac{q p_{\yng(1)}}{p_{\yng(2,2)}}.\]
\end{eg}

The superpotential of the flag variety is a regular function $Y(n,\underline{r})^\circ \times \CC^\rho \to \CC$. The denominators of $W_P$ are the frozen variables of the Grassmannian factors. Any choice of cluster seed for each Grassmannian factor gives a torus chart on $Y(n,\underline{r})^\circ$; pulling back gives a Laurent polynomial with positive coefficients in $\dim(\Fl(n;\underline{r}))$ variables (as in \cite{MarshRietsch}, we normalize by setting $p^i_\emptyset=1$). In particular, this gives many conjectural Laurent polynomial mirrors to the flag variety. 
\begin{mydef} We call the cluster chart on $Y(n,\underline{r})$ given by the rectangles cluster on each Grassmannian factor the \emph{rectangles cluster} for $Y(n,\underline{r})$. 
\end{mydef}
\section{Comparison with the Batyrev--Ciocan-Fontanine--Kim--van Straten mirror}\label{sec:comparison}
The aim of this section is to show that the Pl\U cker coordinate mirror of the flag variety -- whose structure was deduced from quantum Schubert calculus -- pulls back in the rectangles cluster chart to the Eguchi--Hori--Xiong (EHX) mirror \cite{eguchi,giventalflag,flagdegenerations}. This was shown for the Grassmannian in \cite{MarshRietsch}. The Newton polytope of the BCFKvS mirror is a toric degeneration of the flag variety. 

The BCFKvS mirror is defined via the ladder diagram. See for example \cite{flagdegenerations} for complete definition of the BCFKvS mirror; we give a rough description here. Fixing $\Fl(n;r_1,\dots,r_\rho)$, for each Grassmannian step draw an $r_i \times (r_{i-1}-r_i)$ grid of boxes, placing them together.  For example, the ladder diagram of $\Fl(5,3,2,1)$ is
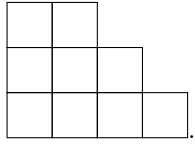
\[\begin{tikzpicture}[scale=0.6]
\draw (0,0) rectangle (1,1);
\draw (1,1) rectangle (2,2);
\draw (0,1) rectangle (1,2);
\draw (1,0) rectangle (2,1);
\draw (2,1) rectangle (3,2);
\draw (0,2) rectangle (1,3);
\draw (1,2) rectangle (2,3);
\draw (2,0) rectangle (3,1);
\draw (3,0) rectangle (4,1);
\end{tikzpicture}.\]
Draw a vertex inside each box in the ladder diagram, as well as vertices in the corners along the northeast border. Add a vertex above the first column of boxes and to the right of the bottom row. In this example, the vertices are
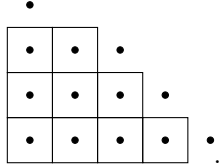
\[\begin{tikzpicture}[scale=0.6]
\draw (0,0) rectangle (1,1);
\draw (1,1) rectangle (2,2);
\draw (0,1) rectangle (1,2);
\draw (1,0) rectangle (2,1);
\draw (2,1) rectangle (3,2);
\draw (0,2) rectangle (1,3);
\draw (1,2) rectangle (2,3);
\draw (2,0) rectangle (3,1);
\draw (3,0) rectangle (4,1);
\draw[fill] (0.5,0.5) circle (2pt);
\draw[fill] (1.5,0.5) circle (2pt);
\draw[fill] (2.5,0.5) circle (2pt);
\draw[fill] (3.5,0.5) circle (2pt);
\draw[fill] (0.5,1.5) circle (2pt);
\draw[fill] (0.5,2.5) circle (2pt);
\draw[fill] (0.5,3.5) circle (2pt);
\draw[fill] (4.5,0.5) circle (2pt);
\draw[fill] (1.5,1.5) circle (2pt);
\draw[fill] (1.5,2.5) circle (2pt);
\draw[fill] (2.5,1.5) circle (2pt);
\draw[fill] (2.5,2.5) circle (2pt);
\draw[fill] (3.5,1.5) circle (2pt);
\end{tikzpicture}.\]
The final step is to add arrows: these are added between all vertices downwards and to the right:
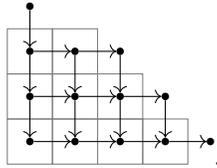
\[\begin{tikzpicture}[scale=0.6]
\draw[gray] (0,0) rectangle (1,1);
\draw[gray] (1,1) rectangle (2,2);
\draw[gray] (0,1) rectangle (1,2);
\draw[gray] (1,0) rectangle (2,1);
\draw[gray] (2,1) rectangle (3,2);
\draw[gray] (0,2) rectangle (1,3);
\draw[gray] (1,2) rectangle (2,3);
\draw[gray] (2,0) rectangle (3,1);
\draw[gray] (3,0) rectangle (4,1);
\node[circle,fill,scale=0.3] at (0.5,0.5) (a) {};
\node[circle,fill,scale=0.3] at (1.5,0.5) (b) {};
\node[circle,fill,scale=0.3] at (2.5,0.5) (c) {};
\node[circle,fill,scale=0.3] at (3.5,0.5) (d) {};
\node[circle,fill,scale=0.3] at (0.5,1.5) (e) {};
\node[circle,fill,scale=0.3] at (0.5,2.5) (f) {};
\node[circle,fill,scale=0.3] at (0.5,3.5) (g) {};
\node[circle,fill,scale=0.3] at (4.5,0.5) (h) {};
\node[circle,fill,scale=0.3] at (1.5,1.5) (i) {};
\node[circle,fill,scale=0.3] at (1.5,2.5) (j) {};
\node[circle,fill,scale=0.3] at (2.5,1.5) (k) {};
\node[circle,fill,scale=0.3] at (2.5,2.5) (l) {};
\node[circle,fill,scale=0.3] at (3.5,1.5) (m) {};
\draw[<-] (h)--(d);
\draw[<-] (d)--(c);
\draw[<-] (c)--(b);
\draw[<-] (b)--(a);
\draw[<-] (m)--(k);
\draw[<-] (k)--(i);
\draw[<-] (i)--(e);
\draw[<-] (l)--(j);
\draw[<-] (j)--(f);

\draw[->] (e)--(a);
\draw[->] (f)--(e);
\draw[->] (g)--(f);
\draw[->] (i)--(b);
\draw[->] (j)--(i);
\draw[->] (k)--(c);
\draw[->] (l)--(k);
\draw[->] (m)--(d);
\end{tikzpicture}.\]
Assign to each of the internal vertices a variable $z_v$. To the $\rho+1$ external vertices, assign $q_0:=1,q_1,\dots,q_\rho$ from upper left to bottom right. The BCFKvS mirror is
\[\sum_{a} \frac{z_{h(a)}}{z_{t(a)}}\]
where the sum is over the arrows of the quiver.

\begin{prop}[\cite{MarshRietsch}]\label{prop:grassmannian}
The BCFKvS mirror of the Grassmannian is isomorphic to the pullback of the Pl\U cker coordinate mirror along the rectangles cluster chart.
\end{prop}
\begin{proof}[Sketch]
Because this is the key to the main result of this section, we briefly sketch the proof. The Grassmannian $\Gr(n,r)$ is a one step flag variety; there is a single $r \times (n-r)$ grid of blocks in the ladder diagram. To the vertex in the $i^{th}$ row and $j^{th}$ column, assign $\frac{p_{i \times j}}{p_{i-1 \times j-1}}$, where as before $i \times j$ denotes the rectangular partition with $i$ rows and $j$ columns. This describes the change of coordinates inducing the isomorphism between the BCFKvS mirror and the Pl\U cker coordinate mirror in the rectangles cluster chart. In particular, after setting $z_v:=\frac{p_{i \times j}}{p_{i-1 \times j-1}}$, the BCFKvS mirror can be simplified using Pl\U cker relations until it matches $W_P$.
\end{proof}
\begin{eg}
The ladder diagram for $\Gr(4,2)$ is
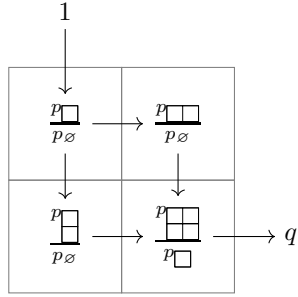
\[\begin{tikzpicture}[scale=1.5]
\draw[gray] (0,0) rectangle (1,1);
\draw[gray] (1,1) rectangle (2,2);
\draw[gray] (0,1) rectangle (1,2);
\draw[gray] (1,0) rectangle (2,1);
\node at (0.5,1.5) (A) {$\frac{p_{\yng(1)}}{p_\emptyset}$};
\node at (0.5,0.5) (B) {$\frac{p_{\yng(1,1)}}{p_\emptyset}$};
\node at (1.5,1.5) (C) {$\frac{p_{\yng(2)}}{p_\emptyset}$};
\node at (1.5,0.5) (D) {$\frac{p_{\yng(2,2)}}{p_{\yng(1)}}$};
\node at (0.5,2.5) (E) {$1$};
\node at (2.5,0.5) (F) {$q$};
\draw[->] (A)--(B);
\draw[->] (B)--(D);
\draw[->] (A)--(C);
\draw[->] (C)--(D);
\draw[->] (E)--(A);
\draw[->] (D)--(F);
\end{tikzpicture}.\]
The contributions to the superpotential from the two middle horizontal arrows are
\[\frac{p_{\yng(2)}}{p_{\yng(1)}}+\frac{p_{\yng(2,2)}p_\emptyset}{p_{\yng(1)}p_{\yng(1,1)}}=\frac{p_{\yng(2)}p_{\yng(1,1)}+p_{\yng(2,2)}p_\emptyset}{p_{\yng(1)} p_{\yng(1,1)}}=\frac{p_{\yng(2,1)}}{p_{\yng(1,1)}},\]
where the last equality is by the Pl\U cker relation
\[p_{\yng(2)}p_{\yng(1,1)}+p_{\yng(2,2)}p_\emptyset = p_{\yng(2,1)}p_{\yng(1)}. \]
Similarly, summing up the other arrows results in $W_P$. 
\end{eg}
We now prove the main theorem, relating the Pl\U cker coordinate superpotential and the BCFKvS mirror. We normalize by setting $p^i_\emptyset=1$. As in the Grassmannian case, we define the isomorphism of tori by giving a new label for each vertex of the ladder diagram. Fix a flag variety $\Fl(n;r_1,\dots,r_\rho)$. The ladder diagram is made up of the ladder diagrams of $\rho$ Grassmannians, i.e. an $r_i \times (r_{i-1} - r_i)$ grid for each $i$. Label the vertices (both internal and external) of the block corresponding to $\Gr(r_{i-1},r_i)$ with the Pl\U cker coordinates of the $i^{th}$ factor of $Y(n,\underline{r})=\prod_{i=1}^\rho \Gr(r_{i-1},r_{i-1}-r_i)$ just as prescribed in the Grassmannian case, but scale all labels by $q_1 \cdots q_{i-1}$. For example, for $\Fl(5;3,2,1)$ the labels are
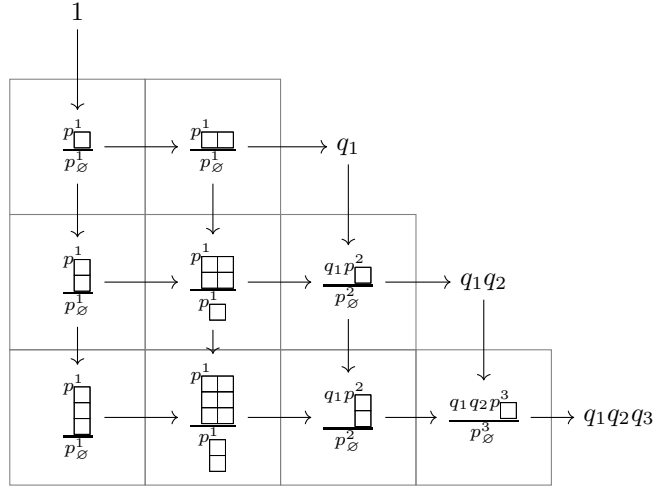
\[\begin{tikzpicture}[scale=1.8]
\draw[gray] (0,0) rectangle (1,1);
\draw[gray] (1,1) rectangle (2,2);
\draw[gray] (0,1) rectangle (1,2);
\draw[gray] (1,0) rectangle (2,1);
\draw[gray] (2,1) rectangle (3,2);
\draw[gray] (0,2) rectangle (1,3);
\draw[gray] (1,2) rectangle (2,3);
\draw[gray] (2,0) rectangle (3,1);
\draw[gray] (3,0) rectangle (4,1);
\node at (0.5,0.5) (a) {$\frac{p^1_{\yng(1,1,1)}}{p^1_\emptyset}$};
\node at (1.5,0.5) (b) {$\frac{p^1_{\yng(2,2,2)}}{p^1_{\yng(1,1)}}$};
\node at (2.5,0.5) (c) {$\frac{q_1  p^2_{\yng(1,1)}}{p^2_\emptyset}$};
\node at (3.5,0.5) (d) {$ \frac{q_1 q_2 p^3_{\yng(1)}}{p^3_\emptyset}$};
\node at (0.5,1.5) (e) {$\frac{p^1_{\yng(1,1)}}{p^1_\emptyset}$};
\node at (0.5,2.5) (f) {$\frac{p^1_{\yng(1)}}{p^1_\emptyset}$};
\node at (0.5,3.5) (g) {1};
\node at (4.5,0.5) (h) {$q_1 q_2 q_3$};
\node at (1.5,1.5) (i) {$\frac{p^1_{\yng(2,2)}}{p^1_{\yng(1)}}$};
\node at (1.5,2.5) (j) {$\frac{p^1_{\yng(2)}}{p^1_\emptyset}$};
\node at (2.5,1.5) (k) {$\frac{q_1  p^2_{\yng(1)}}{p^2_\emptyset}$};
\node at (2.5,2.5) (l) {$q_1$};
\node at (3.5,1.5) (m) {$q_1 q_2$};
\draw[<-] (h)--(d);
\draw[<-] (d)--(c);
\draw[<-] (c)--(b);
\draw[<-] (b)--(a);
\draw[<-] (m)--(k);
\draw[<-] (k)--(i);
\draw[<-] (i)--(e);
\draw[<-] (l)--(j);
\draw[<-] (j)--(f);

\draw[->] (e)--(a);
\draw[->] (f)--(e);
\draw[->] (g)--(f);
\draw[->] (i)--(b);
\draw[->] (j)--(i);
\draw[->] (k)--(c);
\draw[->] (l)--(k);
\draw[->] (m)--(d);
\end{tikzpicture}.\]
Viewing the variables of the BCFKvS mirror and the rectangular Pl\U cker coordinates on $Y(n; \underline{r})$ as coordinates on the torus $(\CC^*)^{\dim(\Fl(n;r_1,\dots,r_\rho))}$, the labeling above gives an automorphism 
\[\phi: (\CC^*)^{\dim(\Fl(n;r_1,\dots,r_\rho))} \to (\CC^*)^{\dim(\Fl(n;r_1,\dots,r_\rho))}.\]
\begin{thm} Let $\Fl(n;r_1\dots,r_\rho)$ be a flag variety, and let $W_{T}$ denote the Batyrev--Ciocan-Fontanine--Kim--van Straten mirror (T standing for toric variety). Let $W_R$ denote the pull-back of the Pl\U cker coordinate superpotential along the rectangles chart. Then
\[\phi^*(W_T)=W_R.\]
\end{thm}
\begin{proof}
By Proposition \ref{prop:grassmannian}, by using Pl\U cker relations, $\phi^*(W_T)$ is almost the sum of the Pl\U cker coordinate mirrors of each Grassmannian fiber in $\Fl(n;\underline{r})$. This is also true of the Pl\U cker coordinate mirror of the flag variety. We just need to compare the discrepancies. Call the rational function on $Y(n,\underline{r})$ that is the sum of the Pl\U cker coordinates of each Grassmannian fiber $W_F$. 

Looking first at $\phi^*(W_T)$, the only arrows that do not induce a summand that appears in $W_F$ are the horizontal arrows that cross from one block to another. The other arrows give all of the summands of $W_F$, except for the $\rho$ that involve the quantum parameters. Consider the arrows in the ladder diagram going from the $i^{th}$ block to the $i+1^{th}$, $i<\rho$. For $r_{i-1} \geq j \geq r_{i-1}-r_i$, we have an arrow giving a summand
\begin{equation} \label{eq:extra}
q_{i}\frac{p^{i}_{(j-1) \times (r_{i}-r_{i+1}-1)} p^i_{(j-(r_{i}-r_{i+1}))^t}}{p^{i}_{j \times (r_{i}-r_{i+1})}}.\end{equation}
Now consider the Pl\U cker coordinate mirror, as determined by Lemma \ref{lem:formula}. For each $\lambda \in M(r_{i-1},r_i)$, we take the equation $p^i_{\yng(1)}*p^i_\lambda=G^i_\lambda$, solve for $p^i_{\yng(1)}$, and sum the results, then subtract off $r_{i+1}$ copies of $p^i_{\yng(1)}$.  If $\lambda$ falls into cases 2, 3, or 4 in the lemma, the term in $G^i_\lambda/p^i_{\yng(1)}$ involving Pl\U cker coordinates from the $i-1^{th}$ factor cancels with subtracted terms. The terms in cases 1, 2, 3, and 4 that do not involve in the quantum parameters all appear in $W_F$: again, the only summands in $W_F$ we miss are those involving the quantum parameter. To prove the theorem it suffices to compare the extra terms with the summands in $\phi^*(W_T)$ described in \eqref{eq:extra}. But this immediate from looking at the formula for $G^i_\lambda$. 
\end{proof}
\section{Examples}\label{sec:examples}
We have compared the Pl\U cker coordinate mirror with the Batyrev--Ciocan-Fontanine--Kim--van Straten mirror. Rietsch \cite{rietsch2006} carefully describes the relation between the BCFKvS mirror and the Lie theoretic mirror, which can be combined with the above result to relate indirectly $W_P$ and the Lie theoretic construction. We don't explore this here, but instead, discuss in some examples the relation between the Gu--Sharpe mirror and the Pl\U cker coordinate mirror. In the Grassmannian case, we describe the isomorphism identifying the critical loci of the two mirrors. We also discuss the family $\Fl(n;2,1)$. To do this, we need the description of the quantum cohomology ring of the flag variety given by the Abelian/non-Abelian correspondence. 

\subsection{The Abelian/non-Abelian correspondence and the quantum cohomology ring of a flag variety} The quantum cohomology of any type A flag variety can be written as 
\[\CC[x_{ij},q_i: 1 \leq i \leq \rho, 1 \leq j \leq r_i]^{\Sym_{r_1} \times \cdots \times \Sym_{r_\rho}}/I.\] 
Recall that $\Sym_{r_i}$ acts on the polynomial ring by permuting the $x_{i1},\dots,x_{i r_i}$.  By convention, we set $x_{0j}=0$. The ideal $I$ can be described via the Abelian/non-Abelian correspondence: as shown in \cite{gukalashnikov}, one can describe the relations in the quantum cohomology ring by considering the relations in the quantum cohomology ring of a Fano toric variety instead. This is the toric variety obtained by Abelianization.  Let $\omega:=\prod_{i=1}^\rho \prod_{j < k} (x_{ij}-x_{ik})$, and set $I^{ab}$ to be the ideal in  $\CC[x_{ij},q_i: 1 \leq i \leq \rho, 1 \leq j \leq r_i]$ generated by the relations, for $i=1,\dots,\rho$, $j=1,\dots,r_i$
\[\prod_{k=1}^{r_{i-1}} (x_{ij}-x_{ik})=(-1)^{r_i-1}q_i \prod_{k=1}^{r_{i+1}} (x_{ik}-x_{ij}).\]
 Then the statement is that $f$ is in $I$ if and only if $\omega f \in I^{ab}$. In \cite{gukalashnikov}, we show how to use this ideal to compute structure constants of the quantum cohomology ring of the flag variety for a particular basis (not the Schubert basis). 
 
 This is the natural basis to consider from the perspective of the A/NA correspondence. To describe it, let $S^i_\lambda:=[s_\lambda(x_{i1},\dots,x_{i r_i})]$ be the class  in $QH(\Fl(n;\underline{r}))$ of the Schur polynomial indexed by a partition $\lambda$ in the $i^{th}$ set of variables. The basis is the set
 \[\{S^1_{\mu_1} \cdots S^\rho_{\mu_\rho}: \underline{\mu} \in S(n;\underline{r})\}.\]
Computations in this basis are done by first multiplying using the Littlewood--Richardson rules, and reducing any partitions which are too wide using rim--hook removals: see \cite[Example 3.8]{gukalashnikov} for the exact formula. 

The Gu--Sharpe superpotential \cite{gusharpe} is a Laurent polynomial in the $x_{ij}$ -- it is obtained taking the superpotential of the Abelianization of the flag variety and specializing the quantum parameters. We do not need the exact formula for the superpotential, just its critical locus.  The critical locus  of the Gu--Sharpe mirror is cut out by $I^{ab}$, with the additional restriction that for all $j \neq k$, $x_{ij} \neq x_{ik}$. This is a discrete set of points, on which there is a free $\Sym_{r_1} \times \cdots \times \Sym_{r_\rho}$ action. We consider the critical locus up to this action, which exactly describes the quantum cohomology ring of the flag variety. In the family of examples we consider, we show that $W_P$ has the correct critical locus by showing a bijective map between the two critical loci. Let us first consider the simpler case of the Grassmannian. 
\begin{eg}\label{eg:grassmanniancrit} Consider the one-step flag variety $\Gr(n,r)$. The critical locus of the Gu--Sharpe superpotential is described by 
\[x_1^n= \cdots =x_r^n=(-1)^r q.\]
In \cite{karp}, Karp shows that the critical locus of the Pl\U cker coordinate mirror is the $\binom{n}{r}$ fixed points of a q-deformation of the cyclic shift map (the author thanks Lauren Williams for pointing this paper out to her). He also shows that these fixed points are points of the Grassmannian corresponding to the matrices 
\[\begin{bmatrix}
1 & x_1 & x_1^2 & \cdots & x_1^{n-1}\\
1 & x_2 & x_2^2 & \cdots & x_2^{n-1}\\
\vdots& \vdots & & \vdots\\
1 & x_r & x_r^2 & \cdots & x_r^{n-1}\\
\end{bmatrix}\]
 where $x_1^n= \cdots=x_r^n=(-1)^r q$ and (so that this describes a point of the Grassmannian) $x_i \neq x_j$. After quotienting by the $S_r$ action, this gives an identification between the critical loci of the two superpotentials. Equivalently, we can describe the fixed locus of $W_P$ as the points given by
 \[p_\lambda=s_\lambda(x_1,\dots,x_r), \; x_1^n= \cdots =x_r^n=(-1)^r q.\]
\end{eg}
\begin{eg}[$\Fl(n;2,1)$]
We look at the family of examples $F_n:=\Fl(n;2,1)$. Let $Y_n:=Y(n;2,1)$. The critical locus of the Gu--Sharpe mirror for $F_n$ is given by the equations
\[x_{11}^n=-q_1 (x_{21}-x_{11}), \; x_{12}^n=-q_1 (x_{21}-x_{12}), \; (x_{21}-x_{11}) (x_{21}-x_{11})=q_2,\; x_{11} \neq x_{12}.\]
For fixed $q_1,q_2$, this gives $2 \binom{n}{2}$ critical points, up to the $\Sym_2$ action permuting $x_{11}$ and $x_{12}$. Call this set $C$. As in the Grassmannian case, using the Pl\U cker embeddings of the factors of $Y_n$ allows us to specify points of $Y$ by specifying the values of the Pl\U cker coordinates. Let 
\[C_P:=\{p^1_\lambda=s_\lambda(x_{11},x_{12}), p^2_{\yng(1)}=q_2/s_{\yng(1)}(x_{21})=q_2/x_{21}\}.\]
This is well-defined when $q_1^2 \neq q_2^{n-1}$ and $q_1 q_2 \neq 0$. When $q_1^2 = q_2^{n-1}$, the number of critical points of the Pl\U cker coordinate mirror can drop (for example, when $q_1=q_2=1$ for $\Fl(4;2,1)$, the critical locus has only 11 points).  Note that $C_P \subset Y_n^\circ$ and $|C_P|=2 \binom{n}{2}$.

So suppose $q_1^2 \neq q_2^{n-1}, q_1 q_2 \neq 0$. We expect that $C_P$ is the critical locus of $W_P$, which we have verified in small examples. We now check that the intersection of $C_P$ with the rectangles torus lies in the critical locus of $W_R$ (the pullback of $W_P$ to this chart). When $C_P$ is contained in this torus, this completely describes the critical locus of $W_P$. What we need to verify is that, for every rectangular $\lambda$, the equation
\[\frac{\partial}{\partial p^i_\lambda}W_R=0\]
gives a valid quantum cohomology relation, setting $p^1_\lambda=S^1_\lambda$, $p^2_{\yng(1)}=q_2/S^2_{\yng(1)}$. We can do this explicitly in this simple case. For example, we can compute that
\[p^2_{\yng(1)} \frac{\partial}{\partial p^2_{\yng(1)}}W_R=\frac{q_1 p^1_{(n-3)} p^2_{\yng(1)}}{p^1_{(n-2,n-2)}}+p^2_{\yng(1)}-\frac{q_2}{p^2_{\yng(1)}}.\]
We thus need to check the quantum cohomology ring relation
\[\frac{q_1 q_2 S^1_{(n-3)}}{S^1_{(n-2,n-2)} S^2_{\yng(1)}}-S^2_{\yng(1)}+\frac{q_2}{S^2_{\yng(1)}}=0.\]
That is, we want to check that
\begin{equation}\label{eq:goal} q_1 q_2 S^1_{(n-3)}-(S^2_{\yng(1)})^2 S^1_{(n-2,n-2)}+q_2 S^1_{(n-2,n-2)}=0.\end{equation}
Using the rim-hook removal rule from \cite{gukalashnikov}, one can compute that
\[(S^2_{\yng(1)})^2=q_2+ S^1_{\yng(1)}  S^2_{\yng(1)}- S^1_{\yng(1,1)}.\] 
Using Littlewood-Richardson rules,
\[(S^2_{\yng(1)})^2 S^1_{(n-2,n-2)}=q_2 S^1_{(n-2,n-2)}+S^1_{(n-1,n-2)}  S^2_{\yng(1)}-S^1_{(n-1,n-1)}.\]
Equation \eqref{eq:goal} then follows from applying the rim-hook removal to compute
\[S^1_{(n-1,n-2)}  S^2_{\yng(1)}=q_1 S^2_{\yng(1)}(S^1_{(n-3)} S^2_{\yng(1)}- S^1_{(n-3,1)})=q_1(q_2 S^1_{(n-3)}+S^1_{(n-2)}S^2_{\yng(1)}-S^1_{(n-2,1)}),\]
\[S^1_{(n-1,n-1)}=q_1(S^1_{(n-2)} S^2_{\yng(1)}- S^1_{(n-2,1)}).\]
To finish the proof of the proposition, one must also explicitly check 
\[\frac{\partial}{\partial p^1_{(n-2)}}W_R=\frac{\partial}{\partial p^1_{(n-2,n-2)}}W_R=\frac{\partial}{\partial p^1_{(n-3)}}W_R=0,\]
which is done in the same way as above. The remaining rectangular partitions only appear in summands arising from arrows entirely contained in the first block of the ladder diagram. The corresponding relation is true for any Grassmannian $\Gr(n,2)$, and thus is true simply as a statement about Schur polynomials. 
\end{eg}
\bibliographystyle{amsplain}
\bibliography{bibliography}
\end{document}